\documentclass[11pt]{article}

\usepackage[T1]{fontenc}
\usepackage[utf8]{inputenc}
\usepackage{lmodern}
\usepackage{amsmath,amssymb,amsthm,mathtools,url}
\usepackage{enumitem}
\usepackage{booktabs}
\usepackage{longtable}
\usepackage{geometry}
\usepackage[colorlinks=true,linkcolor=blue,citecolor=blue,urlcolor=blue]{hyperref}
\usepackage[normalem]{ulem}

\newtheorem{theorem}{Theorem}[section]
\newtheorem{proposition}[theorem]{Proposition}
\newtheorem{lemma}[theorem]{Lemma}

\newtheorem*{theoremA}{Theorem A}
\theoremstyle{remark}
\newtheorem{remark}[theorem]{Remark}

\newcommand{\PP}{\mathbb P}
\newcommand{\C}{\mathbb C}

\newcommand{\Z}{\mathbb Z}
\newcommand{\A}{\mathcal A}
\newcommand{\D}{\Delta}

\title{A counterexample to Purdy's inequality for hyperplane arrangements in projective three-space}
\author{Mateusz Michałek and Piotr Pokora}
\date{\today}

\begin{document}

\maketitle

\begin{abstract}
In the present note we show that Purdy's inequality
$$
        p-\ell+n+2\geq 0
$$
for essential arrangements of $n$ planes in $\mathbb{P}^{3}_{\mathbb{C}}$ with $\ell$
intersection lines and $p$ intersection points fails even when its classical
obstruction, namely a pair of pencils of planes with skew axes, or dually a
point configuration contained in two skew lines, is excluded. Our counterexample
consists of $12$ planes and has $\ell=58$ and $p=43$, so $p-\ell+n+2=-1$.
It is the $\mu_{3}$-balanced $K_{2,2}$-subarrangement of the reflection
arrangement of $G(3,3,4)$, and it emerges naturally once we check that every
irreducible reflection arrangement in $\mathbb{P}^{3}_{\mathbb{C}}$ satisfies Purdy's
inequality, with equality only for $B_{4}$, $G(3,3,4)$, and $G(4,4,4)$.
Moreover, in the analogous family of $\mu_{m}$-balanced configurations the case
$m=3$ is the only one that violates the inequality. Finally, we present an
elementary family of point sets in $\mathbb{P}^{3}$, not contained in two lines nor in
a plane and a point, for which the ratio between the numbers of spanned planes
and spanned lines tends to $3/4$, which improves the value $5/6$ obtained
by Lund.
\end{abstract}



\section{Introduction}

Let \(\A\) be an arrangement of \(n\) hyperplanes in the complex projective
space \(\PP^3_{\C}\).  We write \(\ell=\ell(\A)\) for the number of distinct
intersection lines of the arrangement, and \(p=p(\A)\) for the number of
distinct intersection points, where by an intersection point we mean a point
which is the intersection of at least three hyperplanes of \(\A\) with linearly
independent defining forms, i.e., a \(0\)-dimensional flat of the intersection
lattice of \(\A\). In particular, points lying on an intersection line, but
not cut out by the hyperplanes passing through them, are not counted.  The numerical expression
\[
        p(\A)-\ell(\A)+n+2
\]
is the projective-arrangement counterpart of Purdy's point-line-plane problem in
three-dimensional incidence geometry.  In the dual point language it becomes
\[
        f_2(S)-f_1(S)+f_0(S)+2,
\]
where \(S\subset (\PP^3_{\C})^\vee\) is the dual point configuration and
\(f_i(S)\) denotes the number of \(i\)-dimensional projective flats spanned by
\(S\).  For more details we refer to Purdy's works \cite{Purdy1981, Purdy1986, PurdyMFO}, to the survey by Erd\H{o}s and Purdy
\cite{ErdosPurdy1995}, and to the later incidence-theoretic work of Lund \cite{Lund2018}.

A natural motivation for Purdy's inequality comes from the enumerative
geometry of flats spanned by finite point configurations. This is a classical
topic, with such celebrated results as the de Bruijn--Erd\H{o}s theorem on the
number of lines spanned by a finite set of points in the plane
\cite{deBruijnErdos1948}, and at the same time a very active area of research,
see for instance a very recent counterexample to Mason's conjecture on the
log-concavity of the numbers of flats of a given rank in a matroid \cite{Matt}.
In our notation, the instance of Mason's conjecture concerning the numbers of
points, lines, and planes, known as the points-lines-planes conjecture, reads
\(f_{1}(S)^{2}\geq f_{0}(S)f_{2}(S)\).
In a setting related to ours, Huh and Wang proved the top-heavy conjecture
of Dowling and Wilson for realizable matroids \cite{Huh}: if \(E\) is a finite
subset spanning a vector space \(V\) of dimension \(d\), and \(w_{k}\) denotes
the number of \(k\)-dimensional subspaces of \(V\) spanned by subsets of \(E\),
then \(w_{k}\leq w_{d-k}\) for every \(k\leq d/2\). In \(\PP^{3}\), i.e., for
\(d=4\), this gives \(f_{0}(S)\leq f_{2}(S)\), but it says nothing about the
middle number \(f_{1}(S)\) of spanned lines.  Purdy's proposed inequality may be viewed as a
sharper, low-dimensional numerical refinement of this top-heavy principle,
which also takes the middle term into account:
for a point configuration \(S\subset \mathbb P^3\), or equivalently for
the dual arrangement of hyperplanes, it asks whether the defect
\[
        f_2(S)-f_1(S)+f_0(S)+2
\]
is always non-negative. A first tempting formulation is that every essential arrangement in \(\PP^3\)
should satisfy
\begin{equation}\label{eq:purdy-naive}
        p(\A)-\ell(\A)+n+2\geq 0.
\end{equation}
This is false.  The standard counterexample is a rank \(2+2\) product: take
\(a\) hyperplanes in one pencil and \(b\) hyperplanes in another pencil with
skew axis.  Dually, this is the configuration of \(a\) points on one line and
\(b\) points on a skew line.  This gives
\[
        n=a+b,
        \qquad
        \ell=ab+2,
        \qquad
        p=a+b,
\]
and therefore
\[
        p-\ell+n+2=2a+2b-ab,
\]
which is negative whenever \((a-2)(b-2)>4\).

This classical obstruction suggests the refined statement obtained by excluding
point configurations contained in the union of two skew lines.  The main aim
of the present note is to show that this refinement is also false. Our
counterexample is small and highly structured, and it arises naturally when one
looks for extremal examples among reflection arrangements and their
subarrangements. It is a \(\mu_3\)-balanced \(K_{2,2}\)-configuration inside the monomial reflection arrangement of type
\(G(3,3,4)\). More precisely, let \(\zeta\) be a primitive third root of unity and consider
\[
        S=\{[e_i-\zeta^a e_j]\mid i\in\{0,1\},\ j\in\{2,3\},\ a\in \Z/3\Z\}
        \subset \PP^3_{\C}.
\]
Our main result can be formulated as follows.

\begin{theoremA}
Let \(\A_{\mu_{3}}\subset\PP^{3}_{\C}\) be the arrangement of \(12\) planes
\[
        z_{i}-\zeta^{a}z_{j}=0,
        \qquad
        i\in\{0,1\},\ j\in\{2,3\},\ a\in\Z/3\Z ,
\]
whose dual point configuration is \(S\). Then
\[
        (f_0(S),f_1(S),f_2(S))=(n,\ell(\A_{\mu_{3}}),p(\A_{\mu_{3}}))=(12,58,43),
\]
hence
\[
        p(\A_{\mu_{3}})-\ell(\A_{\mu_{3}})+n+2=-1 .
\]
Moreover, \(\A_{\mu_{3}}\) is essential and \(S\) is not contained in the union
of two skew lines.
\end{theoremA}

Thus our example is a counterexample not only to the naive essential version of
Purdy's inequality, but also to the refined version in which only the
two-skew-lines obstruction is excluded. Let us also emphasize that the third
roots of unity play a special role here: among the analogous
\(\mu_{m}\)-balanced configurations, the case \(m=3\) is the only one with
negative defect, see Section~\ref{subsec:critical}.

Let us also recall another conjecture of Purdy, which asserts that if the
number of points \(f_0(S)\) is sufficiently large, then
\begin{itemize}
    \item \(S\) can be covered by two lines, or by a plane and a point, or
    \item $f_2(S)\geq f_1(S)$.
\end{itemize}
This conjecture was disproved by Lund \cite{Lund2018}, who constructed
infinite families with \(f_2(S)/f_1(S)\to 5/6\). In
Section~\ref{sec: planes vs lines} we present an elementary construction for
which this ratio tends to \(3/4\). To the best of our knowledge, this is the
smallest asymptotic value known so far.

The paper is organized as follows. In Section~\ref{sec:duality} we recall the
precise dual dictionary between hyperplane arrangements and point
configurations. In Section~\ref{sec:two-skew-lines} we recall the classical
two-skew-lines obstruction. In Section~\ref{sec:reflection-path} we explain how
reflection arrangements lead us to the new example: all irreducible reflection
arrangements in \(\PP^{3}_{\C}\), including the primitive ones, satisfy the
inequality, with equality only for \(B_{4}\), \(G(3,3,4)\), and \(G(4,4,4)\),
while a natural gain-graphic subarrangement of \(G(3,3,4)\) violates it. In
Section~\ref{sec:mu3-counterexample} we provide the complete count for the
\(\mu_3\)-balanced \(K_{2,2}\)-configuration, which proves Theorem~A, and we
explain why \(m=3\) is critical. Finally, in Section~\ref{sec: planes vs lines}
we present an elementary construction for which, asymptotically, the number of
spanned planes is only \(\frac{3}{4}\) of the number of spanned lines.

\section{Dual formulation}
\label{sec:duality}

Let \(S\subset \PP^3_{\C}\) be a finite point set.  We put
\[
        f_0(S)=|S|,
\]
\[
        f_1(S)=\#\{\text{lines spanned by pairs of points of }S\},
\]
and
\[
        f_2(S)=\#\{\text{planes spanned by triples of non-collinear points of }S\}.
\]
We also write
\[
        \D(S)=f_2(S)-f_1(S)+f_0(S)+2
\]
and call \(\D(S)\) the \textbf{Purdy defect} of \(S\).

Let
\[
        \A=\{H_1,\ldots,H_n\}
\]
be an arrangement of hyperplanes in \(\PP^3_{\C}\).  Let
\[
        S=\{q_1,\ldots,q_n\}\subset (\PP^3_{\C})^\vee
\]
be the dual point configuration, where \(q_i\) corresponds to the hyperplane
\(H_i\).

\begin{proposition}\label{prop:dual-dictionary}
With the notation above,
\[
        n=f_0(S),
        \qquad
        \ell(\A)=f_1(S),
        \qquad
        p(\A)=f_2(S).
\]
Consequently,
\[
        p(\A)-\ell(\A)+n+2=\D(S).
\]
Moreover, \(\A\) is essential if and only if \(S\) spans the whole dual space
\((\PP^3_{\C})^\vee\).
\end{proposition}

\begin{proof}
The equality \(n=f_0(S)\) is immediate.  The intersection line
\(H_i\cap H_j\) corresponds dually to the line \(\langle q_i,q_j\rangle\)
spanned by the two dual points.  Therefore distinct intersection lines of
\(\A\) are in bijection with spanned lines of \(S\), which gives
\(\ell(\A)=f_1(S)\).

Similarly, the hyperplanes passing through a point \(P\in\PP^{3}_{\C}\)
correspond to the points of \(S\) lying on the dual plane \(P^{\vee}\subset
(\PP^{3}_{\C})^{\vee}\). The point \(P\) is an intersection point of \(\A\) if
and only if these hyperplanes have three linearly independent defining forms,
that is, if and only if \(P^{\vee}\cap S\) contains three non-collinear points,
and then \(P^{\vee}\) is the plane spanned by them. This gives a bijection
between intersection points of \(\A\) and planes spanned by \(S\), so
\(p(\A)=f_2(S)\).  The formula for the defect follows.

Finally, the arrangement is essential precisely when the defining linear forms
span the full dual vector space.  Projectively, this says that the points
\(q_i\) span \((\PP^3_{\C})^\vee\).
\end{proof}
Thus Purdy's inequality for arrangements is equivalent to
\[
\D(S)=f_2(S)-f_1(S)+|S|+2\geq 0
\]
for the dual point configuration.

\section{The two-skew-lines obstruction}
\label{sec:two-skew-lines}

Let us recall the standard obstruction.  Let
\[
        S=A\cup B\subset \PP^{3}_{\mathbb{C}},
\]
where \(A\subset L\), \(B\subset M\), and \(L,M\subset \PP^{3}_{\mathbb{C}}\) are two skew
lines.  Write
\[
        |A|=a,
        \qquad
        |B|=b,
\]
and assume \(a,b\geq 2\).  Then \(S\) spans \(\PP^{3}_{\mathbb{C}}\).

\begin{proposition}\label{prop:two-skew-lines}
With the notation above,
\[
        f_1(S)=ab+2,
        \qquad
        f_2(S)=a+b,
\]
and therefore
\[
        \D(S)=2a+2b-ab.
\]
In particular, \(\D(S)<0\) if and only if
\[
        (a-2)(b-2)>4.
\]
\end{proposition}

\begin{proof}
The spanned lines are the two lines \(L\) and \(M\), together with the \(ab\)
joining lines
\[
        \langle a_i,b_j\rangle,
        \qquad
        a_i\in A,\ b_j\in B.
\]
These joining lines are pairwise distinct, since \(\langle a_i,b_j\rangle\)
meets \(L\) only at \(a_i\) and \(M\) only at \(b_j\). Hence \(f_1(S)=ab+2\).

The spanned planes are exactly the planes
\[
        \langle L,b_j\rangle,
        \qquad b_j\in B,
\]
and
\[
        \langle M,a_i\rangle,
        \qquad a_i\in A.
\]
They are pairwise distinct, since a plane containing \(L\) meets \(M\) in
exactly one point, and no plane contains both \(L\) and \(M\).
Thus \(f_2(S)=a+b\).  Since \(|S|=a+b\), we get
\[
        \D(S)=(a+b)-(ab+2)+(a+b)+2=2a+2b-ab.
\]
The final assertion follows by rewriting \(ab>2a+2b\) as
\((a-2)(b-2)>4\).
\end{proof}

Dually, Proposition~\ref{prop:two-skew-lines} gives an essential arrangement of
planes in \(\PP^{3}_{\mathbb{C}}\) violating \eqref{eq:purdy-naive}.  In coordinates
\([x:y:z:w]\), this arrangement has the form
\[
        x-\lambda_i y=0,
        \qquad i=1,\ldots,a,
\]
together with
\[
        z-\mu_j w=0,
        \qquad j=1,\ldots,b,
\]
where all \(\lambda_i\)'s are distinct and all \(\mu_j\)'s are distinct.  It is
an essential rank \(2+2\) product arrangement.  For example \((a,b)=(5,5)\)
gives
\[
        n=10,
        \qquad
        \ell=27,
        \qquad
        p=10,
        \qquad
        p-\ell+n+2=-5.
\]
This is the classical obstruction in Purdy's problem.
It is worth noticing that the smallest number of planes for which this
construction violates \eqref{eq:purdy-naive} is \(n=9\), attained for
\((a,b)=(4,5)\), where the defect is equal to \(-2\).

\section{Reflection arrangements as the path to the counterexample}
\label{sec:reflection-path}

A natural first place to look for examples is the class of reflection
arrangements.  General references for reflection arrangements and complex
reflection groups are Orlik--Terao \cite{OrlikTerao1992} and Shephard--Todd
\cite{ShephardTodd1954}.  We record the outcome of this search, since it is
instructive: the full irreducible rank-four reflection arrangements satisfy
Purdy's inequality, while the counterexample appears as a natural subarrangement
of an imprimitive reflection arrangement. We refer to Hunt's PhD thesis \cite[Section 2.1]{Hunt} for the numerical data describing reflection arrangements listed below; all values were additionally verified by a direct computer count.

\subsection{Irreducible Coxeter arrangements}

The irreducible finite Coxeter arrangements of rank four are
\[
        A_4,\quad D_4,\quad B_4,\quad F_4,\quad H_4.
\]
Their values are as follows:
\begin{center}
\begin{tabular}{ccccc}
\toprule
Type & \(n\) & \(\ell\) & \(p\) & \(p-\ell+n+2\) \\
\midrule
\(A_4\) & 10 & 25  & 15   & 2 \\
\(D_4\) & 12 & 34  & 24   & 4 \\
\(B_4\) & 16 & 58  & 40   & 0 \\
\(F_4\) & 24 & 122 & 120  & 24 \\
\(H_4\) & 60 & 722 & 1320 & 660 \\
\bottomrule
\end{tabular}
\end{center}
Thus no irreducible rank-four Coxeter arrangement violates the inequality.  The
boundary case is \(B_4\), where equality holds.

\subsection{Irreducible monomial complex reflection arrangements}

For \(G(m,k,4)\) with \(m\geq 2\), \(k\mid m\), and \(k<m\) (we write \(k\)
instead of the customary \(p\) in order to avoid a clash with \(p=p(\A)\)), the
reflection arrangement is the full monomial arrangement
\[
        \A(m,1,4):\quad
        z_0z_1z_2z_3
        \prod_{0\leq i<j\leq 3}(z_i^m-z_j^m)=0.
\]
For this arrangement,
\[
        n=6m+4,
        \qquad
        \ell=7m^2+12m+6,
        \qquad
        p=m^3+4m^2+6m+4.
\]
Therefore
\[
        p-\ell+n+2=(m-2)^2(m+1)\geq 0.
\]
Equality occurs at \(m=2\), namely the Coxeter arrangement of type \(B_4\).

For \(G(m,m,4)\), the reflection arrangement is
\[
        \A(m,m,4):\quad
        \prod_{0\leq i<j\leq 3}(z_i^m-z_j^m)=0.
\]
Here
\[
        n=6m,
        \qquad
        \ell=7m^2+6,
        \qquad
        p=m^3+6m+4,
\]
and hence
\[
        p-\ell+n+2=m(m-3)(m-4).
\]
For \(m\geq 2\), this is nonnegative, with equality at \(m=3\) and \(m=4\).
Thus the full irreducible imprimitive arrangements do not give a counterexample.

\subsection{Primitive complex reflection arrangements}

In the classification of Shephard and Todd \cite{ShephardTodd1954}, the
remaining irreducible complex reflection groups of rank four are the primitive
groups \(G_{28}\), \(G_{29}\), \(G_{30}\), \(G_{31}\), and \(G_{32}\).
Since \(G_{28}=W(F_{4})\) and \(G_{30}=W(H_{4})\) were considered above, it
remains to look at \(G_{29}\), \(G_{31}\), and \(G_{32}\). They define the
arrangements denoted by \(a^{3}_{1}(40)\), \(a^{3}_{2}(60)\) (Klein's \(60_{15}\)
configuration), and \(a^{3}_{2}(40)\) in \cite[Section 2.1]{Hunt},
respectively. Their values are as follows:
\begin{center}
\begin{tabular}{ccccc}
\toprule
Type & \(n\) & \(\ell\) & \(p\) & \(p-\ell+n+2\) \\
\midrule
\(G_{29}\) & 40 & 310 & 380  & 112 \\
\(G_{31}\) & 60 & 710 & 1500 & 852 \\
\(G_{32}\) & 40 & 330 & 400  & 112 \\
\bottomrule
\end{tabular}
\end{center}
Summing up, every irreducible reflection arrangement in \(\PP^{3}_{\C}\)
satisfies Purdy's inequality, and equality holds exactly for \(B_{4}\),
\(\A(3,3,4)\), and \(\A(4,4,4)\).

\subsection{The gain-graphic subarrangement suggested by \texorpdfstring{\(G(3,3,4)\)}{G(3,3,4)}}

The equality case \(\A(3,3,4)\) is especially suggestive.  Its defining
polynomial is
\[
        \prod_{0\leq i<j\leq 3}(z_i^3-z_j^3),
\]
so its hyperplanes are
\[
        z_i-\zeta^a z_j=0,
        \qquad
        0\leq i<j\leq 3,
        \quad
        a\in \Z/3\Z.
\]
Although the full arrangement has defect zero, a counterexample appears as
soon as we keep only the hyperplanes indexed by the four edges of the complete
bipartite graph
\[
        K_{2,2}:\quad \{0,1\}\times \{2,3\}.
\]
Thus we consider the subarrangement
\[
        z_i-\zeta^a z_j=0,
        \qquad
        i\in\{0,1\},
        \quad
        j\in\{2,3\},
        \quad
        a\in\Z/3\Z.
\]
This is a \(\mu_3\)-gain-graphic arrangement supported on a four-cycle.  The
balanced-cycle relation in this four-cycle is exactly the mechanism that reduces
the number of spanned planes and produces a negative Purdy defect.  For the
general language of biased and gain graphs, see Zaslavsky \cite{Zaslavsky1989}.

\section{The \texorpdfstring{\(\mu_3\)}{mu3}-balanced \texorpdfstring{\(K_{2,2}\)}{K22} counterexample}
\label{sec:mu3-counterexample}

Let \(\zeta\) be a primitive third root of unity.  In \(\PP^3_{\C}\), with
homogeneous coordinates \(X_0,X_1,X_2,X_3\) corresponding to the standard basis \(e_0,e_1,e_2,e_3\), define
\[
        p_{ij}^a=[e_i-\zeta^a e_j],
        \qquad
        i\in\{0,1\},
        \quad
        j\in\{2,3\},
        \quad \text{and} \quad
        a\in\Z/3\Z.
\]
Here and below we identify \((\PP^{3}_{\C})^{\vee}\) with \(\PP^{3}_{\C}\)
by means of the dual basis, so that \(p_{ij}^{a}\) is the point dual to the
plane \(z_{i}-\zeta^{a}z_{j}=0\).
Let
\[
        S=\{p_{ij}^a\mid i\in\{0,1\},\ j\in\{2,3\},\ a\in\Z/3\Z\}.
\]
Then \(|S|=12\).  The set is supported on four coordinate lines
\[
        L_{02},\quad L_{03},\quad L_{12},\quad L_{13},
\]
where
\[
        L_{ij}=\PP(\langle e_i,e_j\rangle).
\]
Each \(L_{ij}\) contains the three points
\[
        p_{ij}^0,
        \quad
        p_{ij}^1,
        \quad
        p_{ij}^2.
\]

\begin{lemma}\label{lem:mu3-spans}
The set \(S\) spans \(\PP^3_{\C}\), and it is not contained in the union of two
skew lines.
\end{lemma}

\begin{proof}
The vectors
\[
        e_0-e_2,
        \qquad
        e_0-\zeta e_2,
        \qquad
        e_0-e_3,
        \qquad
        e_1-e_2
\]
span \(\C^4\), thus \(S\) spans \(\PP^3_{\C}\).

The set \(S\) is contained in four distinct lines \(L_{02},L_{03},L_{12},L_{13}\), each
containing three points of \(S\).  If a finite set were contained in the union of
two skew lines, then any line containing three of its points would have to be one
of those two lines.  Here there are four such lines.  Hence \(S\) is not
contained in the union of two skew lines.
\end{proof}

\begin{lemma}\label{lem:mu3-lines}
The number of lines spanned by \(S\) is
\[
        f_1(S)=58.
\]
\end{lemma}

\begin{proof}
There are \(\binom{12}{2}=66\) unordered pairs of points.  The only collinear
triples in \(S\) are the four triples lying on
\[
        L_{02},\quad L_{03},\quad L_{12},\quad L_{13}.
\]
Indeed, observe that every point of \(S\) has exactly two non-zero coordinates.
Let \(u\in L_{ij}\) and \(v\in L_{kl}\) be two points of \(S\) lying on
different lines, and suppose that a third point \(w\in S\) lies on the line
\(\langle u,v\rangle\), i.e., \(w=\lambda u+\mu v\) with \(\lambda\mu\neq 0\).
If \(\{i,j\}\cap\{k,l\}=\emptyset\), then \(\lambda u+\mu v\) has four non-zero
coordinates, a contradiction. If the two lines share an index, then
\(\lambda u+\mu v\) has exactly two non-zero coordinates only if its coordinate
at the shared index vanishes, and then \(w\) lies on the line spanned by the
two remaining basis vectors. These two indices are either \(\{0,1\}\) or
\(\{2,3\}\), and \(S\) has no points on \(L_{01}\) or \(L_{23}\), a
contradiction. For instance, for \(u=p_{02}^{a}\) and \(v=p_{03}^{b}\) we get
\[
        \lambda u+\mu v=(\lambda+\mu)e_0-\lambda\zeta^{a}e_2-\mu\zeta^{b}e_3 ,
\]
which, for \(\lambda+\mu=0\), is a point of \(L_{23}\). This is precisely the
place where the bipartite structure matters: the missing edges \(\{0,1\}\) and
\(\{2,3\}\) of \(K_{2,2}\) carry no points of \(S\).

On each of the four lines \(L_{ij}\), three pairs of points give only one
spanned line.  Thus each such line reduces the naive count by
\[
        \binom{3}{2}-1=2.
\]
Therefore
\[
        f_1(S)=\binom{12}{2}-4\left(\binom{3}{2}-1\right)=66-8=58.
\]
\end{proof}

\begin{lemma}\label{lem:mu3-planes}
The number of planes spanned by \(S\) is
\[
        f_2(S)=43.
\]
\end{lemma}

\begin{proof}
Let \(P\) be a plane spanned by \(S\). If \(P\) contains two points of a line
\(L_{ij}\), then \(L_{ij}\subset P\). Observe that the two pairs of opposite
lines \((L_{02},L_{13})\) and \((L_{03},L_{12})\) consist of skew lines, while
two lines sharing an index span a coordinate face plane. Hence \(P\) contains
either two, or one, or none of the lines \(L_{ij}\), and this gives three
pairwise disjoint types of spanned planes, which we are going to count
separately.

First, there are the four coordinate face planes
\[
        \PP(\langle e_0,e_2,e_3\rangle),
        \quad
        \PP(\langle e_1,e_2,e_3\rangle),
\]
\[
        \PP(\langle e_0,e_1,e_2\rangle),
        \quad
        \PP(\langle e_0,e_1,e_3\rangle).
\]
Each contains two of the four lines \(L_{ij}\).  This gives \(4\) planes.

Second, take one of the four lines \(L_{ij}\) and one point on the opposite
line.  The opposite pairs are
\[
        (L_{02},L_{13})
        \qquad\text{and}\qquad
        (L_{03},L_{12}).
\]
If a plane \(P\) contains, say, \(L_{02}\) but no other line \(L_{ij}\),
then \(P\cap L_{03}=\{e_0\}\) and \(P\cap L_{12}=\{e_2\}\), and these points do
not belong to \(S\). Hence all the remaining points of \(S\) on \(P\) lie on
the opposite line \(L_{13}\), and there is exactly one such point.
For each of the four lines there are three choices of a point on the opposite
line.  This gives
\[
        4\cdot 3=12
\]
planes.

Third, consider planes containing none of the lines \(L_{ij}\). Such a plane
contains at most one point of each line \(L_{ij}\), and hence at least three
points lying on three different lines. We are going to show that it contains
exactly one point from each of the four lines. Take
\[
        p_{02}^a,
        \quad
        p_{03}^b,
        \quad
        p_{12}^c,
        \quad
        p_{13}^d.
\]
These four points are coplanar if and only if
\[
        a+d\equiv b+c \pmod 3.
\]
Indeed, using the representatives
\[
        p_{02}^a=(1,0,-\zeta^a,0),
        \qquad
        p_{03}^b=(1,0,0,-\zeta^b),
\]
\[
        p_{12}^c=(0,1,-\zeta^c,0),
        \qquad
        p_{13}^d=(0,1,0,-\zeta^d),
\]
we compute
\[
\det
\begin{pmatrix}
1&0&-\zeta^a&0\\
1&0&0&-\zeta^b\\
0&1&-\zeta^c&0\\
0&1&0&-\zeta^d
\end{pmatrix}
=
\zeta^{a+d}-\zeta^{b+c}.
\]
Thus the determinant vanishes precisely when
\[
        \zeta^{a+d}=\zeta^{b+c},
\]
or equivalently
\[
        a+d\equiv b+c\pmod 3.
\]
For any choice of \(a,b,c\in\Z/3\Z\), there is a unique
\[
        d=b+c-a
\]
satisfying this congruence.  Hence there are
\[
        3^3=27
\]
planes of this third type.  These planes are distinct: after normalizing the
coefficient of \(X_2\) to be \(1\), the plane has equation
\[
        \zeta^a X_0+\zeta^c X_1+X_2+\zeta^{a-b}X_3=0,
\]
which determines the triple \((a,b,c)\).
Moreover, all coefficients of this equation are non-zero, so the plane
contains none of the points \(e_0,\ldots,e_3\), and hence none of the lines
\(L_{ij}\); it meets \(L_{13}=\{X_0=X_2=0\}\) exactly at the point
\(p_{13}^{\,b+c-a}\in S\). The transpositions of coordinates \((0\,1)\) and
\((2\,3)\) preserve \(S\) and act transitively on the four lines \(L_{ij}\), so
the same holds for any choice of three of the four lines. Consequently, every
plane spanned by \(S\) which contains none of the lines \(L_{ij}\) contains
exactly one point of each line \(L_{ij}\), and it is one of the \(27\) planes
above.

Since the three types above are pairwise disjoint and exhaust all spanned
planes, we obtain
\[
        f_2(S)=4+12+27=43.
\]
\end{proof}
Now we are ready to prove our main result, which we formulate in the
language of hyperplane arrangements.
\begin{theorem}[The \(\mu_3\)-balanced counterexample]\label{thm:mu3-counterexample}
Let \(\zeta\) be a primitive third root of unity and let
\[
        \A_{\mu_3}
        =
        \{z_i-\zeta^a z_j=0
        \mid
        i\in\{0,1\},\ j\in\{2,3\},\ a\in\Z/3\Z\}
        \subset \PP^3_{\C}.
\]
Then \(\A_{\mu_3}\) is an essential arrangement of \(12\) hyperplanes.  It has
\[
        \ell(\A_{\mu_3})=58,
        \qquad
        p(\A_{\mu_3})=43.
\]
Consequently,
\[
        p(\A_{\mu_3})-\ell(\A_{\mu_3})+12+2=-1.
\]
In particular, Purdy's inequality fails for \(\A_{\mu_3}\), although its
dual point configuration is not contained in the union of two skew lines.
\end{theorem}

\begin{proof}
The dual point configuration of \(\A_{\mu_3}\) is exactly the set \(S\) above.
By Proposition~\ref{prop:dual-dictionary},
\[
        n=f_0(S),
        \qquad
        \ell(\A_{\mu_3})=f_1(S),
        \qquad
        p(\A_{\mu_3})=f_2(S).
\]
By Lemmas~\ref{lem:mu3-spans}, \ref{lem:mu3-lines}, and
\ref{lem:mu3-planes}, we have
\[
        f_0(S)=12,
        \qquad
        f_1(S)=58,
        \qquad
        f_2(S)=43.
\]
Thus
\[
        p(\A_{\mu_3})-\ell(\A_{\mu_3})+12+2=43-58+12+2=-1.
\]
The arrangement is essential because its dual point configuration spans
\(\PP^3_{\C}\), and \(S\) is not contained in two skew lines by
Lemma~\ref{lem:mu3-spans}. This completes the proof.
\end{proof}

\begin{remark}
The mechanism is the balanced four-cycle relation
\[
        a+d\equiv b+c\pmod 3.
\]
It says that, once three points on three edges of the \(K_{2,2}\)-configuration
are chosen, there is a unique fourth point on the remaining edge such that the
four points are coplanar.  This forces many four-point planes and therefore
reduces the number of distinct spanned planes.
\end{remark}

\subsection{Why the third root of unity is critical for the counterexample}
\label{subsec:critical}

The same construction can be made with \(\mu_m\).  Let \(\omega\) be a primitive
\(m\)-th root of unity with \(m\geq 2\) and set
\[
        S_m=\{[e_i-\omega^a e_j]\mid i\in\{0,1\},\ j\in\{2,3\},\ a\in\Z/m\Z\}.
\]
Then
\[
        f_0(S_m)=4m.
\]
The proofs of Lemmas~\ref{lem:mu3-lines} and \ref{lem:mu3-planes} carry over
verbatim. The four coordinate lines each contain \(m\) points, and there are no other
collinear triples.  Hence
\[
        f_1(S_m)=\binom{4m}{2}-4\left(\binom m2-1\right)=6m^2+4.
\]
The planes are counted as above: four coordinate face planes, \(4m\) planes
containing one coordinate line and one point on the opposite coordinate line,
and \(m^3\) balanced four-point planes.  Hence
\[
        f_2(S_m)=m^3+4m+4.
\]
Therefore
\[
        \D(S_m)=m^3-6m^2+8m+2\;=\;m(m-2)(m-4)+2.
\]
For \(m=2,3,4\) one obtains respectively
\[
        2,
        \quad
        -1,
        \quad
        2.
\]
For \(m\geq 5\), the product \(m(m-2)(m-4)\) is positive, and so is
\(\D(S_m)\).  Thus the case \(m=3\) is the
unique critical root-of-unity case in this family.

\section{Counting lines and planes}\label{sec: planes vs lines}
It is very interesting to compare the ratio
\[
        r(S):=\frac{f_2(S)}{f_1(S)}
\]
of the number of planes to the number of lines determined by a finite set
\(S\subset\PP^{3}\). This question has attracted considerable attention,
especially in the asymptotic regime when \(|S|\) tends to infinity. Recall
that a set of points \(S\) is called \emph{irreducible} if it is contained
neither in the union of two lines nor in the union of a plane and a point.
In \cite{ErdosPurdy1995} Purdy conjectured that \(r(S)\geq 1\) for every
irreducible set \(S\) with \(|S|\) large enough, see also
\cite[Section 3.5]{LundPurdySmith2011}. The conjecture was disproved in a very
strong sense by Lund \cite{Lund2018}, who showed that
\(\liminf_{|S|\rightarrow \infty} r(S)\leq 5/6\), where the limit inferior is
taken over irreducible sets \(S\). On the other hand, it has been known for
many years \cite{Purdy1986} that this limit inferior is strictly positive.
Below we present a very simple construction showing that
\(\liminf_{|S|\rightarrow \infty} r(S)\leq 3/4\), which in particular provides
further irreducible counterexamples to \(f_2(S)\geq f_1(S)\).

\begin{proposition}
    For any $a,b\geq 3$ there exists an irreducible set $S$ of $a+b+1$ points such that $f_1(S)=ab+a+b$ and $f_2(S)=ab+1$. In particular, taking $a=3$ and $b\rightarrow\infty$ we obtain
    \[\liminf_{|S|\rightarrow \infty} r(S)\leq 3/4,\]
    where the limit inferior is taken over irreducible sets.
\end{proposition}
\begin{proof}
    Pick two skew lines $l_1,l_2$, $a$ distinct points $p_1,\ldots,p_a$ on
    $l_1$, and $b$ distinct points $q_1,\ldots,q_b$ on $l_2$. We define $S$ by
    adding one point $x$ on the line $\langle p_1,q_1\rangle$, different from
    $p_1$ and $q_1$.

    By Proposition~\ref{prop:two-skew-lines} the points $p_i$ and $q_j$ span
    $ab+2$ lines. Every new line passes through $x$ and one of the points
    $p_i$ or $q_j$, and the line through $x$, $p_1$, $q_1$ has already been
    counted. The remaining lines $\langle x,p_i\rangle$ with $i\neq 1$ and
    $\langle x,q_j\rangle$ with $j\neq 1$ are new and pairwise distinct. Indeed,
    $\langle x,p_i\rangle$ meets $l_1$ only at $p_i$, it lies in the plane
    $\langle l_1,q_1\rangle$, which meets $l_2$ only at $q_1$, and it does not
    pass through $q_1$, since otherwise it would coincide with
    $\langle p_1,q_1\rangle$; symmetrically for $\langle x,q_j\rangle$. Thus we obtain $(a-1)+(b-1)$ new lines, which gives
    \[f_1(S)=ab+2+a-1+b-1=ab+a+b.\]

    Similarly, we have $a+b$ planes spanned by the points $p_i$ and $q_j$. These include the two planes $\langle l_1,x\rangle=\langle l_1,q_1\rangle$ and $\langle l_2,x\rangle=\langle l_2,p_1\rangle$. Hence every new plane is spanned by $x$, $p_i$, $q_j$ with $i,j\neq 1$. Such a plane $P_{ij}$ does not contain $p_1$, since otherwise it would contain $\langle p_1,x\rangle\ni q_1$, and hence the points $p_1,p_i\in l_1$ and $q_1,q_j\in l_2$, i.e., both skew lines. By symmetry, $P_{ij}$ does not contain $q_1$ either, so $P_{ij}$ meets $l_1\cup l_2$ exactly in $\{p_i,q_j\}$. Consequently, the planes $P_{ij}$ are new and pairwise distinct. There are $(a-1)(b-1)$ such planes, and we obtain
    \[f_2(S)=a+b+(a-1)(b-1)=ab+1.\]
    The set $S$ is irreducible. Indeed, since $a,b\geq 3$, any two lines covering $S$ must be $l_1$ and $l_2$, which do not contain $x$, and a plane containing all but one point of $S$ contains at least two points on each of $l_1$ and $l_2$, which is impossible. For $a=3$ we obtain
    \[\lim_{b\rightarrow\infty} \frac{3b+1}{4b+3}=\frac{3}{4}.\]
\end{proof}

\begin{remark}
For a fixed \(a\geq 3\) and \(b\rightarrow\infty\) the above construction gives
\(r(S)\rightarrow a/(a+1)\), so the choice \(a=3\) is the optimal one in this
family.
\end{remark}

\section*{Acknowledgments}
Mateusz Micha{\l}ek is supported by the DFG Grant \textbf{580118961}.
Piotr Pokora is supported by the National Science Centre (Poland) Sonata Bis Grant
\textbf{2023/50/E/ST1/00025}. For the purpose of Open Access, the author has applied a CC-BY public copyright licence to any Author Accepted Manuscript (AAM) version arising from this submission. We acknowledge the use of \texttt{ChatGPT} in the generation of computer programs that helped us to conduct symbolic computations for configurations of points.

\bigskip

Mateusz Michałek \\
\noindent
Department of Mathematics and Statistics,
University of Konstanz, Post Box 192, D-78457
Konstanz, Germany\\
Email: \url{mateusz.michalek@uni-konstanz.de}

\bigskip
Piotr Pokora\\
\noindent
Department of Mathematics,
University of the National Education Commission Krakow,
Podchor\c a\.zych 2,
PL-30-084 Krak\'ow, Poland. \\
Email: \url{piotr.pokora@uken.krakow.pl}

\begin{thebibliography}{99}

\bibitem{deBruijnErdos1948}
N. G. de Bruijn and P. Erd\H{o}s, On a combinatorial problem.
\textit{Indag. Math.} \textbf{10}: 421 -- 423 (1948).

\bibitem{ErdosPurdy1995}
P. Erd\H{o}s and G. B. Purdy, Extremal problems in combinatorial geometry.
In: R. L. Graham, M. Gr\"otschel, and L. Lov\'asz (eds.), \textit{Handbook of
Combinatorics}, Vol. 1, 809 -- 873. Amsterdam: Elsevier (1995).

\bibitem{Huh}
J. Huh and B. Wang, Enumeration of points, lines, planes, etc. \textit{Acta Math.} \textbf{218(2)}: 297 -- 317 (2017).

\bibitem{Hunt}
B. Hunt, Coverings and ball quotients with special emphasis on the 3-dimensional case. \textit{Bonn. Math. Schr.} \textbf{174}: 87 p. (1986).

\bibitem{Matt}
M. Larson, Counterexamples to two conjectures about matroids. \textbf{arXiv:2607.02208} (2026).

\bibitem{Lund2018}
B. Lund, Essential dimension and the flats spanned by a point set. \textit{Combinatorica} \textbf{38}: 1149 -- 1174 (2018).

\bibitem{LundPurdySmith2011}
B. D. Lund, G. B. Purdy, and J. W. Smith, A bichromatic incidence bound and
an application. \textit{Discrete Comput. Geom.} \textbf{46(4)}: 611 -- 625 (2011).

\bibitem{OrlikTerao1992}
P. Orlik and H. Terao, \textit{Arrangements of hyperplanes}. Grundlehren der
Mathematischen Wissenschaften \textbf{300}. Berlin: Springer-Verlag (1992).

\bibitem{Purdy1981}
G. B. Purdy, A proof of a consequence of Dirac’s conjecture.
\textit{Geometriae Dedicata} \textbf{10}: 317 -- 321 (1981).

\bibitem{Purdy1986}
G. B. Purdy, Two results about points, lines and planes.
\textit{Discrete Mathematics} \textbf{60}: 215 -- 218 (1986).

\bibitem{PurdyMFO}
G. B. Purdy, When is the number of hyperplanes determined by $n$ points in $d$-space
at least the number of $(d-2)$-dimensional flats? Abstract of the talk in MFO Report 44/2011, 2472 -- 2474 (2011), \texttt{DOI:10.4171/OWR/2011/44}.

\bibitem{ShephardTodd1954}
G. C. Shephard and J. A. Todd, Finite unitary reflection groups.
\textit{Canadian Journal of Mathematics} \textbf{6}: 274 -- 304 (1954).

\bibitem{Zaslavsky1989}
T. Zaslavsky, Biased graphs. I. Bias, balance, and gains.
\textit{Journal of Combinatorial Theory, Series B} \textbf{47}: 32 -- 52 (1989).

\end{thebibliography}
\end{document}